\documentclass[11pt,reqno]{article}
\usepackage{lmodern}
\usepackage[T1]{fontenc}
\usepackage[utf8]{inputenc}
\usepackage[a4paper]{geometry}
\usepackage{color}
\usepackage{mathtools}
\usepackage{amsmath}
\usepackage{amsthm}
\usepackage{amssymb}
\usepackage{stmaryrd}
\usepackage{setspace}
\usepackage{microtype}
\usepackage[unicode=true,
 bookmarks=false,
 breaklinks=false,pdfborder={0 0 1},backref=section,colorlinks=true]
 {hyperref}
\hypersetup{
 linkcolor=blue}

\makeatletter
\numberwithin{equation}{section}
\numberwithin{figure}{section}
\theoremstyle{plain}
\newtheorem{thm}{\protect\theoremname}[section]
\theoremstyle{remark}
\newtheorem{rem}[thm]{\protect\remarkname}
\theoremstyle{plain}
\newtheorem{cor}[thm]{\protect\corollaryname}
\theoremstyle{definition}
\newtheorem{example}[thm]{\protect\examplename}
\theoremstyle{plain}
\newtheorem{prop}[thm]{\protect\propositionname}
\theoremstyle{plain}
\newtheorem{lem}[thm]{\protect\lemmaname}

\usepackage[pdftex]{graphicx}
\usepackage{float}\usepackage{verbatim}
\usepackage{mathscinet}

\newcommand{\N}{\mathbb{N}}
\newcommand{\R}{\mathbb{R}}
\newcommand{\E}{\mathbb{E}}
\newcommand{\I}{\mathbb{I}}
\renewcommand{\P}{\mathbb{P}}

\newcommand{\HS}{\mathrm{HS}}

\newcommand{\df}{\mathrm{d}}

\DeclareMathOperator{\Vol}{Vol}

\DeclareMathOperator{\End}{End}

\DeclareMathOperator{\Cay}{Cay}

\usepackage{scalerel,stackengine}
\stackMath
\newcommand\widecheck[1]{%
\savestack{\tmpbox}{\stretchto{%
  \scaleto{%
    \scalerel*[\widthof{\ensuremath{#1}}]{\kern-.6pt\bigwedge\kern-.6pt}%
    {\rule[-\textheight/2]{1ex}{\textheight}}
  }{\textheight}%
}{0.5ex}}%
\stackon[1pt]{#1}{\scalebox{-1}{\tmpbox}}%
}

\date{}

\makeatother

\providecommand{\corollaryname}{Corollary}
\providecommand{\examplename}{Example}
\providecommand{\lemmaname}{Lemma}
\providecommand{\propositionname}{Proposition}
\providecommand{\remarkname}{Remark}
\providecommand{\theoremname}{Theorem}

\begin{document}
\global\long\def\C{\mathbb{C}}%
\global\long\def\P{\mathbb{P}}%
\global\long\def\E{\mathbb{E}}%
\global\long\def\Cay{\mathrm{Cay}}%
\global\long\def\D{\mathfrak{D}}%
\global\long\def\G{\mathcal{G}}%

\global\long\def\Vol{\mathrm{Vol}}%
\global\long\def\B{\mathcal{B}}%
\global\long\def\df{\mathrm{def}}%
\global\long\def\eqdf{\stackrel{\df}{=}}%
\global\long\def\N{\mathbf{N}}%
\global\long\def\I{\mathcal{I}}%
\global\long\def\d{\mathrm{d}}%
\global\long\def\tr{\mathrm{tr}}%
\global\long\def\End{\mathrm{End}}%
\global\long\def\R{\mathbb{R}}%

\title{Quantum Unique Ergodicity for Cayley graphs of quasirandom groups}
\author{Michael Magee, Joe Thomas, and Yufei Zhao}

\maketitle

\begin{abstract}
A finite group $G$ is called $C$-quasirandom (by Gowers) if all
non-trivial irreducible complex representations of $G$ have dimension
at least $C$. For any unit $\ell^{2}$ function on a finite group
we associate the \emph{quantum probability measure} on the group given
by the absolute value squared of the function. 

We show that if a group is highly quasirandom, in the above sense,
then any Cayley graph of this group has an orthonormal eigenbasis
of the adjacency operator such that the quantum probability measures
of the eigenfunctions put close to the correct proportion of their
mass on suitably selected subsets of the group that are not too small. 

\end{abstract}
{\footnotesize{}\tableofcontents{}}{\footnotesize\par}

\section{Introduction}

The main question of quantum chaos is to what extent `chaotic' features
of the geodesic flow on a manifold (for example, ergodicity, exponential
mixing, etc.) manifest themselves in the corresponding quantized system;
that is, the $L^{2}$ Laplace-Beltrami operator and its eigenvalues
and eigenfunctions. One of the main questions here is whether the
quantum probability measures associated to eigenfunctions of the Laplacian
have unique weak-$*$ limits (semiclassical measures) as the corresponding
eigenvalue tends to infinity. If there is a unique limit, the manifold
is called quantum uniquely ergodic. 

In this paper, we work with \emph{graphs} instead of manifolds and
prove results in the spirit of quantum unique ergodicity for certain
families $\{\G_{i}\}_{i\in\I}$, $\I\subset\N$ of $d$-regular graphs,
with $d\geq3$ fixed. We will always write $\G_{i}$ to refer to such
a family of graphs. We write $V_{i}$ for the vertex set of $\G_{i}$,
let $N\eqdf|V_{i}|$ and assume $N\to\infty$ as $i\to\infty$. Each
$\G_{i}$ has an adjacency matrix that has rows and columns indexed
by $V_{i}$, a $1$ in entry $(x,y)$ if there is an edge between
$x$ and $y$, and 0 otherwise; we view this as an operator on $\ell^{2}(V_{i})$.
\emph{In this paper, $\ell^{2}$ norms will be defined with respect
to the counting measure.}

Given an element $\varphi\in\ell^{2}(V_{i})$ with $\|\varphi\|_{\ell^{2}}=1$,
which will usually be an eigenfunction of the adjacency operator of
$\G_{i}$, we\emph{ }associate to $\varphi$ the \emph{quantum probability
measure}\footnote{From the point of view of quantum mechanics, this is the probability
density function.} $\mu_{\varphi}$ on $V_{i}$ defined by
\[
\mu_{\varphi}\eqdf\sum_{v\in V_{i}}|\varphi(v)|^{2}\delta_{v},
\]
where $\delta_{v}$ is the unit mass atom at $v$. Note that $\|\varphi\|_{\ell^{2}}=1$
implies $\mu_{\varphi}$ is a probability measure.

We will say \emph{quantum unique ergodicity }(QUE) holds for a sequence
of adjacency operator eigenfunctions $\varphi_{i}\in\ell^{2}(V_{i})$
with $\|\varphi_{i}\|_{\ell^{2}}=1$ and a sequence of subsets $A_{i}\subset V_{i}$
if
\[
\mu_{\varphi_{i}}[A_{i}]\to\frac{|A_{i}|}{|V_{i}|}=\frac{|A_{i}|}{N}
\]
as $i\to\infty$. It is very hard in general to establish this bound
for all $A_{i}$, so we will restrict to $A_{i}$ that are not too
small. 

Suppose that $G$ is a finite group and $S$ is a symmetric subset
of $G$, then we will denote the Cayley graph associated to the pair
$(G,S)$ by $\Cay(G,S)$. We write $\hat{G}$ for the equivalence
classes of irreducible representations of $G$, and define
\[
\D(G)\eqdf\min_{(\rho,V)\in\hat{G}-\mathrm{triv}}\dim V;
\]
i.e. the smallest dimension of a non-trivial representation of $G$.
Then in the language of Gowers from \cite{Go08}, $G$ is $\D(G)$-\emph{quasirandom}\footnote{Before the formal naming of this property by Gowers, the property
of a group $G$ being $|G|^{\delta}$-quasirandom was used to prove
eigenvalue bounds in works of Sarnak and Xue \cite{SX} and Bourgain
and Gamburd \cite{BGEXPAND}. }\emph{. }The first main theorem of the paper is the following.
\begin{thm}
\label{thm:main-thm-there-exists} Let $G_{i}$ be finite groups with
$|G_{i}|\stackrel{i\to\infty}{\to}\infty$, $S_{i}\subseteq G_{i}$
be symmetric subsets $(S_{i}=S_{i}^{-1})$, $\G_{i}=\Cay(G_{i},S_{i})$
and $t_{i}>0$. Moreover, let $M_{i}\in\N$ be such that
\begin{equation}
2M_{i}\sum_{(\pi,V)\in\hat{G_{i}}-\mathrm{triv}}(\dim V)^{2}\left(6e^{-\frac{t_{i}\sqrt{\dim V}}{64}}+2e^{-\frac{\dim V}{12}}\right)<1,\label{eq:prob-req}
\end{equation}
 and let $f_{i}^{j}:V_{i}\to\R$ be any collection of functions for
$j=1,...,M_{i}$ and $i\in\N$. Then, there exist orthonormal bases
$\mathcal{B}_{i}$ of $\ell^{2}(G_{i})$ of real-valued eigenfunctions
of $\mathcal{G}_{i}$ such that for every $\varphi\in\mathcal{B}_{i}$
and $j=1,\ldots,M_{i}$, 
\begin{equation}
\left|\mu_{\varphi}[f_{i}^{j}]-\frac{\sum_{g\in G_{i}}f_{i}^{j}(g)}{|G_{i}|}\right|\leq t_{i}\frac{\|f_{i}^{j}\|_{\ell^{2}}}{\sqrt{|G_{i}|}}.\label{eq:ell2-bound}
\end{equation}
If $f_{i}^{j}=\boldsymbol{1}_{A_{i}^{j}}$ for some subsets $A_{i}^{j}\subseteq V_{i}$
then
\begin{equation}
\left|\mu_{\varphi}[A_{i}^{j}]-\frac{|A_{i}^{j}|}{|G_{i}|}\right|\leq t_{i}\frac{\sqrt{|A_{i}^{j}|}}{\sqrt{|G_{i}|}};\label{eq:mainthm-cmplex-eq}
\end{equation}
which in particular implies that $\mu_{\varphi}[A_{i}^{j}]$ is asymptotic
to $\frac{|A_{i}^{j}|}{|G_{i}|}$ as $i\to\infty$ whenever $\frac{t_{i}^{2}|G_{i}|}{|A_{i}^{j}|}=o_{i\to\infty}(1)$.
\end{thm}

\begin{rem}
The proof of Theorem \ref{thm:main-thm-there-exists} is slightly
easier if one only wants complex orthonormal eigenbases; see Remark
\ref{rem:complex-eiengbases} at the end of the paper. In this case,
one can also take the functions $f_{i}^{j}$ to be complex-valued.
\end{rem}

The condition \eqref{eq:prob-req} involving $M_{i}$ and $t_{i}$
displays a dependence between the desired strength of the QUE bound
in \eqref{eq:ell2-bound}, and the number of functions that one simultaneously
wishes the bound to hold for. With knowledge on the size and number
of irreducible representations of the group, one can be more precise
with values for $t_{i}$ and $M_{i}.$ 

The most simple case of this is as follows. For groups with $\mathfrak{D}(G)\geq\log^{2}(|G|)$
one can obtain at least logarithmic improvement in \eqref{eq:ell2-bound}
while taking the number of functions to be polynomial in the size
of the group.
\begin{cor}
\label{cor:que-large-qr}Let $\varepsilon>0$, and suppose that $G$
is a finite group satisfying $\mathfrak{D}(G)\geq\log^{2}(|G|)$.
Moreover, let $S\subseteq G$ be a symmetric subset and $\mathcal{G}=\mathrm{Cay}(G,S)$.
Then given $M\in\N$ satisfying $M\leq\min\left(\frac{1}{24}|G|^{\varepsilon},\frac{1}{8}|G|^{-1}e^{\frac{\mathfrak{D}(G)}{12}}\right),$
and functions $f_{i}:V\to\R$ for $i=1,\ldots,M,$ there exists an
orthonormal basis $\mathcal{B}$ of $\ell^{2}(G)$ of real-valued
eigenfunctions of $\mathcal{G}$ such that for every $\varphi\in\mathcal{B}$
and $i=1,\ldots,M$, 
\begin{equation}
\left|\mu_{\varphi}[f_{i}]-\frac{\sum_{g\in G}f_{i}(g)}{|G|}\right|\leq64\frac{(\varepsilon+1)\log(|G|)}{\sqrt{\mathfrak{D}(G)}}\frac{\|f_{i}\|_{\ell^{2}}}{\sqrt{|G|}}.\label{eq:large_qr_bound}
\end{equation}
\end{cor}

\begin{rem}
The proof of Theorem \ref{thm:main-thm-there-exists} shows that if
e.g. $\D(G_{i})\geq|G_{i}|^{\alpha}$ with $\alpha>0$ as it is in
cases of interest (see below), if we only want to obtain 
\[
\left|\mu_{\varphi}[f_{i}^{j}]-\frac{\sum_{g\in G_{i}}f_{i}^{j}(g)}{|G_{i}|}\right|=o\left(\frac{\|f_{i}^{j}\|_{\ell^{2}}}{\sqrt{|G_{i}|}}\right)
\]
above then we can actually take $M_{i}\geq e^{c|G|^{\beta}}$ for
$c,\beta>0$ depending on $\alpha$, i.e. take the number of functions
$f_{i}^{j}$ to be super-polynomial in $|G_{i}|$. 

\end{rem}

\begin{example}
If $\I$ are the prime numbers, $G_{p}=\mathrm{PSL}_{2}(\mathbb{F}_{p}),$
and $\G_{p}$ are any Cayley graphs of $\mathrm{PSL}_{2}(\mathbb{F}_{p})$
with respect to symmetric generators, then a result of Frobenius gives
\[
\D(\mathrm{PSL}_{2}(\mathbb{F}_{p}))\geq\frac{p-1}{2},
\]
and $|G_{p}|\approx p^{3}$. So in this setting, Theorem \ref{thm:main-thm-there-exists}
gives that for any finite collection $A_{p}^{1},...,A_{p}^{m}\subset V_{p}$
with $|A_{p}^{j}|\gg p^{2+\epsilon}$ , there are real orthonormal
eigenbases of $\ell^{2}(\mathrm{PSL}_{2}(\mathbb{F}_{p}))$ such that
for any elements $\varphi_{p}$ of these bases, 
\begin{align*}
\mu_{\varphi_{p}}[A_{p}^{j}] & =\frac{|A_{p}^{j}|}{|\mathrm{PSL}_{2}(\mathbb{F}_{p})|}\left(1+O(p^{-\epsilon})\right)
\end{align*}
as $p\to\infty$.
\end{example}

When $\mathfrak{D}(G)$ is polynomial in $|G|$, we can also obtain
a quantum unique ergodicity result for partitions of the group into
sets whose sizes are on scales of the order $|G|^{1-\eta}$ for some
$\eta>0$ dependent upon on the size of $\mathfrak{D}(G)$. 
\begin{cor}
\label{cor:que-small-scale}Let $G$ be a finite group, $S\subseteq G$
be a symmetric subset and $\mathcal{G}=\mathrm{Cay}(G,S)$. Suppose
that there exists an absolute constant $s>0$ such that $\mathfrak{D}(G)\geq|G|^{s}$
and let $\eta=s-\varepsilon$ for any $0<\varepsilon<s$. Let $A_{i}\subseteq G$
be a collection of subsets partitioning $G$ with sizes satisfying
$c|G|^{1-\eta}\leq|A_{i}|\leq C|G|^{1-\eta}$ for some absolute constants
$c,C>0$. Then, for $|G|$ sufficiently large (dependent only upon
$c$ and $\eta$) there is an orthonormal eigenbasis $\mathcal{B}$
of the adjacency operator of $G$ such that for every $i$ and every
$\varphi\in\mathcal{B}$,

\[
\left|\mu_{\varphi}[A_{i}]-\frac{|A_{i}|}{|G|}\right|\leq\frac{K\log|G|}{|G|^{\frac{1}{2}\varepsilon}}\frac{|A_{i}|}{|G|},
\]
where $K>0$ is a constant dependent only upon $c$.
\end{cor}

So far we have dealt with groups that are at least $\log^{2}(|G|)$-quasirandom.
One key feature of the condition \eqref{eq:prob-req} is that it enables
us to go beyond $\mathfrak{D}(G)\geq\log^{2}(|G|)$. This pertains
to the important class of examples where $G_{n}$ is either the alternating
group $\mathrm{Alt}(n)$ or the symmetric group $\mathrm{Sym}(n)$.
\begin{prop}
\label{prop:que-perm-groups}Let $G_{n}=\mathrm{Alt}(n)$ or $\mathrm{Sym}(n)$,
 $S_{n}\subseteq G_{n}$ be symmetric subsets and $\mathcal{G}_{n}=\mathrm{Cay}(G_{n},S_{n})$.
Then given $M_{n}\in\N$ satisfying $M_{n}=o_{n\to\infty}(n)$ and
functions $f_{i}^{n}:V_{n}\to\R$ for $i=1,\ldots,M_{n}$, there exists
an orthonormal basis $\mathcal{B}_{n}$ of $\ell^{2}(G_{n})$ of real-valued
eigenfunctions of $\mathcal{G}_{n}$ such that for every $\varphi\in\mathcal{B}_{n}$,
$i=1,\ldots,M_{n}$, and $n$ sufficiently large
\begin{equation}
\left|\mu_{\varphi}[f_{i}]-\frac{\sum_{g\in G_{n}}f_{i}^{n}(g)}{|G_{n}|}\right|\leq192\frac{\log(n)}{\sqrt{n}}\frac{\|f_{i}^{n}\|_{\ell^{2}(G_{n})}}{\sqrt{|G_{n}|}}.\label{eq:permutation_bound}
\end{equation}
\end{prop}

The proof of Theorem \ref{thm:main-thm-there-exists} revolves around
the fact that all eigenspaces of Cayley graphs arise from some irreducible
representation of the group and hence have multiplicities at least
the dimension of this corresponding representation. This leads to
a dichotomy: either the eigenspace is trivial (which we can deal with
directly) or has large dimension if the group is suitably quasirandom.
In the latter case, this allows one to choose a random basis for the
eigenspace using a random matrix of large dimension which is reflected
in the condition \eqref{eq:prob-req}. 

We describe in $\S$\ref{sec:Random-basis-construction} a random
model for real eigenbases of Cayley graphs that arise from products
of the classical compact groups with their Haar measures. This model
was used by Sah, Sawhney, and Zhao in \cite{Sa.Sa.Zh20} to show the
existence of eigenbases of Cayley graphs with close to optimal $\ell^{\infty}$
bounds. What we prove here is the following.
\begin{thm}
\label{thm:main-theorem-random}Let $G$ be a finite group, $S\subseteq G$
be a symmetric subset and $\G=\Cay(G,S)$. Let $M\in\N$ and let $f_{1},...,f_{M}\in\ell^{2}(G)$
be a collection of real-valued functions. Then, for any $t>0$, with
probability at least
\[
1-2M\sum_{(\pi,V)\in\hat{G}-\mathrm{triv}}(\dim V)^{2}\left(6e^{-\frac{t\sqrt{\dim V}}{64}}+2e^{-\frac{\dim V}{12}}\right),
\]
if $\B$ is a random real orthonormal eigenbasis of $\G$ as in $\S$\ref{sec:Random-basis-construction},
then for any $\varphi\in\B$ and any $i=1,...,M$, we have 
\begin{equation}
\left|\mu_{\varphi}[f_{i}]-\frac{\sum_{g\in G}f_{i}(g)}{|G|}\right|\leq t\frac{\|f_{i}\|_{\ell^{2}}}{\sqrt{|G|}}.\label{eq:random_thm-cmplex-eq}
\end{equation}
\end{thm}

As indicated by Corollaries \ref{cor:que-large-qr} and \ref{cor:que-small-scale},
it is good to know that there are an abundance of $|G|^{\delta}$-quasirandom
groups for $0<\delta<1$. Indeed, for \textit{finite simple groups
of Lie type} with rank $r$ over finite fields, it is shown in the
proof of \cite[Prop. 3.2]{Br.Gr.Gu.Ta15} (see also Remark 1.3.6 of
\cite{Ta15}) using earlier work of \cite{La.Se74,Se.Za93} that such
groups are $|G|^{\delta}$-quasirandom with $\delta$ depending only
on the rank $r$. We refer to \cite[\S 5.2]{Br.Gr.Gu.Ta15} for the
precise definition of these groups. As such, the values of $t$ in
Theorems \ref{thm:main-thm-there-exists} and \ref{thm:main-theorem-random}
can be taken to have decay that is \emph{polynomial} in $|G|$ for
this wide class of groups (see Corollary \ref{cor:que-large-qr}).

Let us now discuss the strength of the upper bound obtained in Theorems
\ref{thm:main-thm-there-exists} and \ref{thm:main-theorem-random}.
Since the sum of squares of the dimensions of the irreducible representations
of a group equal the size of the group, 
\[
\D(G)\leq|G|^{\frac{1}{2}}
\]
which means that the best possible value we could possibly obtain
for the right hand side of \eqref{eq:mainthm-cmplex-eq} or \eqref{eq:random_thm-cmplex-eq}
is 
\[
\frac{C\log(|G|)}{|G|^{\frac{3}{4}}}\|f\|_{\ell^{2}}.
\]
This is still a factor of $|G|^{\frac{1}{4}}$ off from what is known
about random regular graphs: recently Bauerschmidt, Huang, and Yau
\cite{Ba.Hu.Ya19} obtained a very strong version of QUE for random
regular graphs with respect to the uniform model of fixed degree and
number of vertices\footnote{See also \cite{Ba.Kn.Ya17} for the case of growing degree.}. 
\begin{thm}[{Bauerschmidt-Huang-Yau \cite[Cor. 13]{Ba.Hu.Ya19}}]
 \label{thm:BHYque}Let $d\gg1$ and let $\G_{n}$ be a uniformly
random $d$-regular graph on $n$ vertices. Suppose $f_{n}:V_{n}\to\R$,
then with probability tending to one as $n\to\infty$, for any eigenfunction
$\varphi\in\ell^{2}(V_{n})$ of the adjacency operators of $\G_{n}$
with eigenvalues $\lambda_{n}$ satisfying $|\lambda_{n}\pm2\sqrt{d-1}|>(\log n)^{-\frac{3}{2}}$,
\begin{align*}
\left|\sum_{v\in V_{n}}f_{n}(v)|\varphi(v)|^{2}-\frac{\sum_{v\in V_{n}}f_{n}(v)}{n}\right|\leq\frac{(\log n)^{250}}{n}\sqrt{\sum_{v\in V_{n}}|f_{n}(v)|^{2}}.
\end{align*}
\end{thm}

The first result about equidistribution of quantum probability measures
of eigenfunctions\footnote{Strictly speaking, Theorem \ref{thm:AM} is a result about Quantum
\emph{Ergodicity} rather than QUE.} on graphs was obtained by Anantharaman and Le Masson in \cite[Thm. 1]{An.Le15}. 

\begin{thm}[{\cite[Thm. 1]{An.Le15}}]
\label{thm:AM}Let $\G_{i}$ be $d$-regular, $d>3$, and $N\eqdf|V(\G_{i})|\to\infty$
as $i\to\infty$. Suppose that the sequence $\G_{i}$ form a family
of uniform expanders and converge to the infinite $d$-regular tree
in the sense of Benjamini and Schramm \cite{Be.Sc01}. Let $\{\varphi_{j}^{(i)}\}_{j=1}^{N}$
be an orthonormal basis of eigenfunctions of the adjacency operator
of $\G_{i}$. Let $f_{i}:V_{i}\to\C$ be a sequence of functions with
$\|f_{i}\|_{\infty}\leq1$, then for any $\delta>0$
\begin{align}
\frac{1}{N}\left|\left\{ j\in[1,N]:\left|\sum_{v\in V_{i}}f_{i}(v)|\varphi_{j}^{(i)}(v)|^{2}-\frac{1}{N}\sum_{v\in V_{i}}f_{i}(v)\right|>\delta\right\} \right|\to0\label{eq:ALMqeresult}
\end{align}
as $i\to\infty$.
\end{thm}

For related results of quantum ergodicity on quantum graphs, see for
example \cite{Br.Wi2016,An.In.Sa.Wi21}. See also the recent work
of Naor, Sah, Sawhney and Zhao \cite{Na.Sa.Sa.Zh22} in the Cayley
graph setting, where they prove an incomparable quantum ergodicity
result, rather than quantum unique ergodicity. 

\subsection{QUE on manifolds}

Because the type of results of the current paper draw their inspiration
from analogous questions about manifolds, we include a brief discussion
of the state of the art results in that setting. 

Let $M$ be a closed and connected Riemannian manifold and let $\{\varphi_{j}\}_{j\geq1}$
be an orthonormal basis of $L^{2}(M)$ consisting of Laplacian eigenfunctions
with corresponding eigenvalues $0=\lambda_{1}<\lambda_{2}\leq\ldots\to\infty$.
A central question is the \emph{quantum unique ergodicity conjecture}
of Rudnick and Sarnak \cite{Ru.Sa94}. This says that if $M$ is negatively
curved, then the quantum probability measures of the eigenfunctions
weak-$*$ converge as $i\to\infty$ to the normalized Riemannian volume
form. A more general statement of this conjecture involving microlocal
lifts can be found in the survey article of Sarnak \cite{Sa11}. For
manifolds without negative curvature, there are counterexamples to
this conjecture as illustrated for example by Hassel \cite{Ha10}
for certain ergodic billiards, building upon earlier numerical work
by O'Connor and Heller \cite{O.He88}. 

Despite counterexamples demonstrating that ergodicity alone is insufficient
for quantum unique ergodicity, there is numerical evidence to support
the conjecture in the presence of negative curvature \cite{Au.St93,He.Ra92}.
In addition, there are striking results of Anantharaman and Nonnenmacher
\cite{An.No07,An08} and Dyatlov and Jin \cite{Dy.Ji18} regarding
the entropy and support of possible limits of quantum probability
measures. Moreover, Lindenstrauss \cite{Li06} (with an extension
by Soundararajan \cite{So10} for the non-compact case), proved that
the quantum unique ergodicity conjecture holds for Hecke-Laplace eigenfunctions
on arithmetic surfaces.

For closed Riemannian manifolds in general, ergodicity of the geodesic
flow alone is sufficient to prove a weaker result known as quantum
ergodicity. This result exhibits the existence of a density one subsequence
of the quantum probability measures that weak-$*$ converges to the
normalized volume measure \cite{Sn74,Ze87,Co85}. Theorem \ref{thm:AM}
above can be seen as a natural graph analogue of this weaker property.
In the manifold setting, quantum ergodicity has also been investigated
for random bases. For example, in \cite{Ze1992} it is shown that
random (Haar unitary) eigenbases of the Laplacian for $L^{2}(S^{2})$
are quantum ergodic with probability one, despite the standard basis
of spherical harmonics failing to have this property. This is upgraded
to quantum unique ergodicity in \cite{Va1997}. Similarly, quantum
ergodicity and quantum mixing properties have been studied for random
bases (not necessarily eigenbases) for general compact Riemannian
manifolds \cite{Ze1996,Ze2014} as well as quantum unique ergodicity
\cite{Ma2013}.

\subsection{Outline of the paper}

The remainder of the paper proceeds as follows. In $\S$\ref{sec:Background}
we give an overview of the relevant representation theoretic background
and outline the construction of Cayley graphs and how the adjacency
operator acts through representation theory. In $\S$\ref{sec:Random-basis-construction}
we describe the random bases we use throughout the paper. In $\S$\ref{sec:Deterministic-error-term}
we give a deterministic bound on the quantities 
\[
\left|\mu_{\varphi}[f]-\frac{\sum_{g\in G}f(g)}{|G|}\right|
\]
featuring in the main results. In $\S$\ref{sec:Probabilistic-ingredients}
we give first some basic large deviations estimates for sums of independent
random variables, and then apply these to obtain concentration results
for tensor products of random matrices from the classical compact
groups. Finally, in $\S$\ref{sec:proof-main-thm} we prove Theorem
\ref{thm:main-theorem-random} by combining the deterministic error
estimate and our random matrix results.

\section{Background\label{sec:Background}}

\subsection{Representation theory of finite groups}

We begin by outlining basic concepts in representation theory. A more
complete background can be found in \cite{Fu.Ha91}.

Let $G$ denote a finite group. We consider unitary representations
of $G$. These are pairs $(\pi,V)$ where $V$ is a finite-dimensional
complex Hilbert space and $\pi:G\to GL(V)$ is a homomorphism such
that $\pi(g)$ is unitary for each $g\in G$. When clear, we will
just refer to $\pi$ or $V$ as a representation. We will denote the
trivial representation of $G$ by $(\mathrm{triv},\C)$, where $\C$
has the standard inner product and $\mathrm{triv}(g)$ is the identity
for all $g\in G$.

The group algebra $\C[G]$ is the ring of formal complex linear combinations
of elements of $G$. We identify $\C[G]$ with $\ell^{2}(G)$ throughout
the paper. Any representation $(\pi,V$) of $G$ linearly extends
to $\pi:\C[G]\to\End(V)$ making $V$ a $\C[G]$ module.

Recall that a representation $(\pi,V)$ is irreducible if there are
no proper subspaces of $V$ that are invariant under $\pi(g)$ for
all $g\in G$. Two representations $(\pi_{1},V_{1})$ and $(\pi_{2},V_{2})$
are equivalent if there is a unitary isomorphism $T:V_{1}\to V_{2}$
that intertwines the representations: $T\circ\pi_{1}(g)=\pi_{2}(g)\circ T$
for all $g\in G$. We will denote the unitary dual of $G$ by $\hat{G}$,
it is the collection of equivalence classes of irreducible representations
of $G$. We will not make any distinction between an equivalence class
in $\hat{G}$ and an element of the equivalence class; hence we will
freely write $(\pi,V),\ \pi,\ V\in\hat{G}$.

Given a representation $(\pi,V)$ of $G$, the dual representation
will be denoted by $(\check{\pi},\check{V})$. Here, $\check{V}$
is the dual space of $V$ equipped with the inner product arising
from that of $V$ on the corresponding Riesz representation vectors,
and $\check{\pi}$ is defined by $[\check{\pi}(g)\alpha](v)=\alpha(\pi(g^{-1})v)$
for all $g\in G$ and $v\in V$. If $(\pi,V)$ is irreducible, then
so is $(\check{\pi},\check{V})$.

Given $(\pi,V)\in\hat{G}$, and $v_{1},v_{2}\in V$, the matrix coefficient
\[
\Phi_{v_{1},v_{2}}^{V}\eqdf\langle\pi(g)v_{2},v_{1}\rangle
\]
is in $\ell^{2}(G)$. This extends bilinearly to a map $\Phi^{V}:\check{V}\otimes V\to\ell^{2}(G)$.
The inner product on $\ell^{2}(G)$ is given by
\[
\langle f_{1},f_{2}\rangle\eqdf\sum_{g\in G}f_{1}(g)\overline{f_{2}(g)}.
\]
The space $\ell^{2}(G)$ is a bimodule for $G\times G$ (under left
and right multiplication) and the induced map

\begin{equation}
\Phi\eqdf\bigoplus_{(\pi,V)\in\hat{G}}\frac{\sqrt{\dim V}}{\sqrt{G}}\Phi^{V}:\bigoplus_{(\pi,V)\in\hat{G}}\check{V}\otimes V\to\ell^{2}(G)\label{eq:pter-weyl}
\end{equation}
 is a unitary bimodule isomorphism by the Peter-Weyl theorem. We also
have the Plancherel formula
\begin{equation}
\|f\|_{2}^{2}=\frac{1}{|G|}\sum_{(\pi,V)\in\hat{G}}\dim V\|\pi(f)\|_{\HS}^{2},\label{eq:plancherel}
\end{equation}
where $\|\pi(f)\|_{\HS}^{2}\eqdf\tr_{V}(\pi(f)\pi(f)^{*})$.

\subsection{Cayley graphs}

\label{sec:cayleygraphs} Let $G$ be a finite group and let $S=\{s_{1},s_{1}^{-1},...,s_{d},s_{d}^{-1}\}$
be a symmetric subset in $G$ such that $|S|=2d$. The Cayley graph
$\Cay_{0}(G,S)$ is the directed graph with an edge between $g$ and
$h$ if $gs=h$ for some $s\in S$. The directed edges of $\Cay_{0}(G,S)$
have a pairing arising from matching edges arising from $gs=h$ with
the edge arising from $g=hs^{-1}$; the quotient by this equivalence
relation is the undirected Cayley graph $\Cay(G,S)$, which is a $2d$-regular
graph. The adjacency operator on $\ell^{2}(G)$ can be written as
\begin{align*}
\mathcal{A}[f](g)=\sum_{i=1}^{d}\left(f(gs_{i})+f(gs_{i}^{-1})\right)=\rho(A)[f](g),
\end{align*}
 where $\rho$ is the right regular representation and 
\begin{align*}
A\eqdf\sum_{i=1}^{d}\left(s_{i}+s_{i}^{-1}\right)\in\C[G].
\end{align*}

\section{Random basis construction\label{sec:Random-basis-construction}}

In this section we will outline the construction of the bases of eigenfunctions
for the adjacency operator. The idea is to exploit the decomposition
of $\ell^{2}(G)$ as the direct sum $\bigoplus_{(\pi,V)\in\hat{G}}\check{V}\otimes V$.
To obtain a basis of real-valued functions, one must select the basis
inside each irreducible representation dependent upon whether the
representation is non-self dual, real or quaternionic as we explain
below.

\subsection{Non self-dual representations}

We start with the case that $(\pi,V)$ is an irreducible representation
that is not equivalent to its dual representation $(\check{\pi},\check{V})$.
Due to their non-equivalence, both $\check{V}\otimes V$ and $V\otimes\check{V}$
appear as distinct summands in the decomposition of $\ell^{2}(G)$
as the direct sum $\bigoplus_{(\theta,W)\in\hat{G}}\check{W}\otimes W$.
We will thus seek an orthonormal basis of $(\check{V}\otimes V)\oplus(V\otimes\check{V})$.
As before, let $\{v_{k}^{V}\}$ be an orthonormal basis of $V$ consisting
of eigenvectors of $\pi(A)$. Moreover, let $\{w_{j}^{V}\}$ be any
orthonormal basis of $V$. Then the collection 
\begin{align*}
\left\{ \frac{1}{\sqrt{2}}(\check{w}_{j}^{V}\otimes v_{k}^{V}+w_{j}^{V}\otimes\check{v}_{k}^{V}),\frac{1}{i\sqrt{2}}(\check{w}_{j}^{V}\otimes v_{k}^{V}-w_{j}^{V}\otimes\check{v}_{k}^{V}):j,k=1,\ldots,\dim V\right\} 
\end{align*}
forms an orthonormal basis of $(\check{V}\otimes V)\oplus(V\otimes\check{V})$.
Moreover, they correspond to functions in $\ell^{2}(G)$
\begin{align*}
x_{k,j}^{V}(g) & \eqdf\frac{\sqrt{\dim V}}{\sqrt{2}\sqrt{|G|}}(\langle\pi(g)v_{k}^{V},w_{j}^{V}\rangle+\langle\check{\pi}(g)\check{v}_{k}^{V},\check{w}_{j}^{V}\rangle)=\frac{\sqrt{2\dim V}}{\sqrt{|G|}}\mathrm{Re}(\langle\pi(g)v_{k}^{V},w_{j}^{V}\rangle),\\
y_{k,j}^{V}(g) & \eqdf\frac{\sqrt{\dim V}}{i\sqrt{2}\sqrt{|G|}}(\langle\pi(g)v_{k}^{V},w_{j}^{V}\rangle-\langle\check{\pi}(g)\check{v}_{k}^{V},\check{w}_{j}^{V}\rangle)=\frac{\sqrt{2\dim V}}{\sqrt{|G|}}\mathrm{Im}(\langle\pi(g)v_{k}^{V},w_{j}^{V}\rangle),
\end{align*}
which are real-valued functions with unit $L^{2}$-norm that are mutually
orthogonal.

To randomize this basis, we randomize the choice of the basis $\{w_{j}^{V}\}_{j}$.
We fix an orthonormal basis $\{e_{j}^{V}\}_{j}$ of $V$ and then
given a Haar random unitary operator $u\in U(V)$, we set $w_{j}^{V}=ue_{j}^{V}$
for each $j=1,\ldots,\dim V$.

\subsection{Self-dual representations}

\label{sec:selfdualreps} A complex irreducible representation that
is equivalent to its dual has a conjugate-linear intertwining map
$J:V\to V$ such that $J^{2}=\pm\mathrm{Id}$. In the case $J^{2}=\mathrm{Id}$
the representation is called \emph{real} and in case $J^{2}=\mathrm{-Id}$
the representation is called \emph{quaternionic }\cite{Fu.Ha91}\emph{.
}It is not hard to check using uniqueness (up to scalars) of the $\pi$-invariant
inner product on $V$ that for all $v,w\in V$ 
\begin{equation}
\langle v,w\rangle=\langle J(w),J(v)\rangle.\label{eq:J-inner-product}
\end{equation}

\subsubsection{Real representations. \label{subsec:Real-representations.}}

In this case, $J$ defines a real structure for $V$. That is, $V=V_{J}\oplus iV_{J}$
where $V_{J}=\{v\in V:J(v)=v\}$ is a real vector space. It follows
from \eqref{eq:J-inner-product} that $\langle\bullet,\bullet\rangle$
restricts to a real valued symmetric inner product on $V_{J}$, and
the inner product on $V$ is obtained from this one by extension of
scalars from $\R$ to $\C$.

Since $J$ intertwines with $\pi$, for each $g\in G$ we have $\pi(g):V_{J}\to V_{J}$,
and so $\pi(A)$ is a symmetric operator on $(V_{J},\langle\bullet,\bullet\rangle)$.
Let $\{v_{k}^{V}\}$ denote an orthonormal basis of $\pi(A)$ eigenvectors
in $V_{J}$ with respect to the real inner product. By extension of
scalars, these also form an orthonormal eigenbasis of $\pi(A)$ acting
on $V$.

Fix an orthonormal basis $\{e_{j}^{V}\}$ of $V_{J}$ . Choosing a
Haar random orthogonal matrix $o\in O(V)$ we let $w_{j}^{V}\eqdf oe_{j}^{V}$
for each $1\leq j\leq\dim V$. The corresponding real random basis
of $\rho(A)$ eigenvectors in $\ell^{2}(G)$ is given by 
\begin{align*}
\varphi_{kj}^{V}(g)\eqdf\frac{\sqrt{\dim V}}{\sqrt{|G|}}\langle\pi(g)v_{k}^{V},w_{j}^{V}\rangle.
\end{align*}
These are the image under the inclusion $\check{V}\otimes V\to\ell^{2}(G)$
of the vectors $\check{w}_{j}^{V}\otimes v_{k}^{V}$ (this makes it
clear that they are $\rho(A)$ eigenvectors).

\subsubsection{Quaternionic representations. \label{subsec:Quaternionic-representations.}}

Next, suppose that $(\pi,V)$ is a quaternionic representation of
$G$. In this case, \eqref{eq:J-inner-product} implies
\[
\langle v,J(v)\rangle=\langle J^{2}(v),J(v)\rangle=-\langle v,J(v)\rangle
\]
hence $\langle v,J(v)\rangle=0$ for any $v\in V$. This implies $\dim V$
is even and since $\pi(A)$ is Hermitian and commutes with $J$ we
can find an orthonormal basis of $V$ of eigenvectors of $\pi(A)$
of the form $\{v_{k}^{V},J(v_{k}^{V})\}_{k=1}^{\frac{1}{2}\dim V}$. 

Fix an orthonormal basis $\{e_{j}^{V}\}$ of $V$. Choosing a Haar
random unitary matrix $u\in u(V)$ we let $w_{j}^{V}\eqdf ue_{j}^{V}$
for each $1\leq j\leq\dim V$. The corresponding real random basis
of $\rho(A)$ eigenvectors in $\ell^{2}(G)$ is given by 
\begin{align*}
x_{kj}^{V}(g) & \eqdf\frac{\sqrt{2\dim V}}{\sqrt{|G|}}\mathrm{Re}(\langle\pi(g)v_{k}^{V},w_{j}^{V}\rangle,\\
y_{kj}^{V}(g) & \eqdf\frac{\sqrt{2\dim V}}{\sqrt{|G|}}\mathrm{Im}(\langle\pi(g)v_{k}^{V},w_{j}^{V}\rangle.
\end{align*}
These are the image under the inclusion $\check{V}\otimes V\to\ell^{2}(G)$
of the vectors 
\begin{align*}
x_{kj}^{V} & \eqdf\frac{1}{\sqrt{2}}(\check{w}_{j}^{V}\otimes v_{k}^{V}+\widecheck{J(w_{j}^{V})}\otimes J(v_{k}^{V})),\\
y_{kj}^{V} & \eqdf\frac{1}{i\sqrt{2}}(\check{w}_{j}^{V}\otimes v_{k}^{V}-\widecheck{J(w_{j}^{V})}\otimes J(v_{k}^{V})),
\end{align*}
 and thus clearly they are $\rho(A)$ eigenvectors.

Putting together all of the different cases for the type of the representation
$\pi$, the random model for the real-valued eigenbasis of $\ell^{2}(G)$
has underlying topological space
\begin{align*}
X=\prod_{\substack{\{(\pi,V),(\check{\pi},\check{V})\}\subseteq\hat{G}\\
\pi\ \text{non-self-dual pair}
}
}U(V)\prod_{\substack{(\pi,V)\in\hat{G}\\
\pi\ \text{self-dual and quaternionic}
}
}U(V)\prod_{\substack{(\pi,V)\in\hat{G}\\
\pi\ \text{self-dual and real}
}
}O(V),
\end{align*}
equipped with the product probability measure 
\begin{align}
\P=\prod_{\substack{\{(\pi,V),(\check{\pi},\check{V})\}\subseteq\hat{G}\\
\pi\ \text{non-self-dual pair}
}
}\P_{U(V)}\prod_{\substack{(\pi,V)\in\hat{G}\\
\pi\ \text{self-dual and quaternionic}
}
}\P_{U(V)}\prod_{\substack{(\pi,V)\in\hat{G}\\
\pi\ \text{self-dual and real}
}
}\P_{O(V)},\label{eq:realbasismodel}
\end{align}
where $\P_{U(V)}$ is the Haar probability measure on the unitary
operators $U(V)$ of $V$, and $\P_{O(V)}$ is the Haar probability
measure on the orthogonal operators $O(V)$ of $V$.

\section{Deterministic error term for mean zero functions\label{sec:Deterministic-error-term}}

In this section we will derive an upper bound for 
\begin{align}
\left|\sum_{g\in G}f(g)|\varphi(g)|^{2}-\frac{1}{|G|}\sum_{g\in G}f(g)\right|,\label{eq:quedifference}
\end{align}
where $\varphi$ is one of the eigenbasis elements of $\ell^{2}(G)$
described in the previous section, and $f$ is a real-valued function
on the group $G$. In fact, we will further make the assumption that
\[
\sum_{g\in G}f(g)=0,
\]
so that we can instead just bound
\[
\left|\sum_{g\in G}f(g)|\varphi(g)|^{2}\right|.
\]
This can be done without any loss of generality since given a non-zero
mean function, we can consider $f-\frac{1}{|G|}\sum_{g\in G}f(g)$
which has zero mean, and then a bound on the above quantity for this
zero mean function provides a bound on the desired difference since
\begin{align*}
\left|\sum_{g\in G}\left(f(g)-\frac{1}{|G|}\sum_{h\in G}f(h)\right)|\varphi(g)|^{2}\right| & =\left|\sum_{g\in G}f(g)|\varphi(g)|^{2}-\frac{1}{|G|}\sum_{h\in G}f(h)\sum_{g\in G}|\varphi(g)|^{2}\right|\\
 & =\left|\sum_{g\in G}f(g)|\varphi(g)|^{2}-\frac{1}{|G|}\sum_{g\in G}f(g)\right|,
\end{align*}
as the eigenfunction $\varphi$ is normalized with respect to the
counting measure. The bounds we will obtain later will involve $\|f\|_{\ell^{2}}$,
but since the mean of $f$ is just the Fourier component of $f$ corresponding
to the constant eigenfunction, we have $\|f-\frac{1}{|G|}\sum_{g\in G}f(g)\|_{\ell^{2}}\leq\|f\|_{\ell^{2}}$
and so any bounds depending on the $\ell^{2}$-norm of the zero mean
function can just be bounded by the $\ell^{2}$-norm of the function
itself. 

Now, recall that there were three types of functions in the eigenbasis
dependent upon the type of irreducible representation that they come
from. In the case of complex irreducible representations that are
not real we have the following two types given by real and imaginary
parts of matrix coefficients 
\begin{description}
\item [{Type 1 - Real Part}] 
\begin{align*}
\varphi(g)=\frac{\sqrt{\dim V}}{\sqrt{2}\sqrt{|G|}}(\langle\pi(g)v,w\rangle+\langle\check{\pi}(g)\check{v},\check{w}\rangle),
\end{align*}
\item [{Type 2 - Imaginary Part}] 
\begin{align*}
\varphi(g)=\frac{\sqrt{\dim V}}{i\sqrt{2}\sqrt{|G|}}(\langle\pi(g)v,w\rangle-\langle\check{\pi}(g)\check{v},\check{w}\rangle).
\end{align*}
\end{description}
In the case of a complex irreducible representation that is real we
have the following type of basis element 
\begin{description}
\item [{Type 3 - Real Matrix Coefficient}] 
\begin{align*}
\varphi(g)=\frac{\sqrt{\dim V}}{\sqrt{|G|}}\langle\pi(g)v,w\rangle.
\end{align*}
\end{description}
In each of the above types, $(\pi,V)$ is an irreducible unitary representation
and $v,w\in V$ are unit vectors. 

We will show the following.
\begin{prop}
\label{prop:errorboundreal} Let $\varphi:G\to\R$ be one of types
1,2 or 3, and let $f:G\to\R$ have zero mean. If $\varphi$ is of
type 1 or type 2, then
\begin{align*}
\left|\sum_{g\in G}f(g)|\varphi(g)|^{2}\right|\leq & \left|\mathrm{Re}\left(\left\langle \frac{\mathrm{dim}V}{|G|}\sum_{g\in G}f(g)(\pi\otimes\check{\pi})(g)(v\otimes\check{v}),w\otimes\check{w}\right\rangle \right)\right|\\
 & \hspace{1.5cm}+\left|\mathrm{Re}\left(\left\langle \frac{\mathrm{dim}V}{|G|}\sum_{g\in G}f(g)(\pi\otimes\pi)(g)(v\otimes v),w\otimes w\right\rangle \right)\right|,
\end{align*}
and if $\varphi$ is of type 3, then 
\begin{align*}
\left|\sum_{g\in G}f(g)|\varphi(g)|^{2}\right|\leq & \left|\mathrm{Re}\left(\left\langle \frac{\mathrm{dim}V}{|G|}\sum_{g\in G}f(g)(\pi\otimes\check{\pi})(g)(v\otimes\check{v}),w\otimes\check{w}\right\rangle \right)\right|.
\end{align*}
\end{prop}

\begin{proof}
Suppose that $\varphi$ is type 1. Then, 
\begin{align}
|\varphi(g)|^{2} & =\frac{1}{2}\frac{\dim V}{|G|}\left(\langle(\pi\otimes\pi)(g)(v\otimes v),w\otimes w\rangle+\langle(\check{\pi}\otimes\check{\pi})(g)(\check{v}\otimes\check{v}),\check{w}\otimes\check{w}\rangle\right.\nonumber \\
 & \hspace{1cm}+\left.\langle(\check{\pi}\otimes\pi)(g)(\check{v}\otimes v),\check{w}\otimes w\rangle+\langle(\pi\otimes\check{\pi})(g)(v\otimes\check{v}),w\otimes\check{w}\rangle\right)\nonumber \\
 & =\frac{\dim V}{|G|}\left(\mathrm{Re}\left(\langle(\pi\otimes\pi)(g)(v\otimes v),w\otimes w\rangle\right)+\mathrm{Re}\left(\langle(\pi\otimes\check{\pi})(g)(v\otimes\check{v}),w\otimes\check{w}\rangle\right)\right).\label{eq:splitupphi}
\end{align}
The result is then an immediate application of the triangle inequality
using the fact that $f$ is real-valued. The proof for type 2 functions
is essentially the same, and the proof for type 3 is even simpler
(one only needs to deal with $\pi\otimes\check{\pi}$ terms).
\end{proof}

\section{Probabilistic ingredients\label{sec:Probabilistic-ingredients}}

In this section, we outline some results that we will use in $\S$\ref{sec:proof-main-thm}
when bounding the probability that our random bases have the properties
of Theorems \ref{thm:main-thm-there-exists} and \ref{thm:main-theorem-random}.

\subsection{Large deviations estimates}

We begin by recalling that the $\chi$-squared distribution with $k$-degrees
of freedom, denoted by $\chi_{k}^{2}$, has probability density function
\begin{align}
f_{k}(x)=\frac{x^{\frac{k}{2}-1}e^{-\frac{x}{2}}}{2^{\frac{k}{2}}\Gamma\left(\frac{k}{2}\right)}\mathbf{1}_{\{x>0\}}.\label{eq:pdfchidist}
\end{align}
If $Z_{i},\ldots,Z_{k}$ are independent standard normal random variables,
then 
\begin{align*}
\sum_{i=1}^{k}Z_{i}^{2}\sim\chi_{k}^{2}.
\end{align*}
In this article we will use the following results regarding independent
$\chi_{1}^{2}$ and $\chi_{2}^{2}$ random variables.
\begin{lem}
\label{lem:large norm}If $X_{1},\ldots,X_{N}$ are independent $\chi_{1}^{2}$-distributed
random variables, then
\begin{align*}
\P\left(\sum_{i=1}^{N}X_{i}\leq\frac{N}{2}\right)\leq e^{-\frac{N}{12}}.
\end{align*}
\end{lem}

\begin{proof}
By exponential Chebyshev, for any $A>0$
\begin{align*}
\P\left(\sum_{i=1}^{N}X_{i}\leq t\right) & \leq e^{At}\E\left[e^{-A\sum_{i}X_{i}}\right]=e^{At}\prod_{i=1}^{N}\E\left[e^{-AX_{i}}\right]\\
 & =e^{At}\prod_{i=1}^{N}\frac{1}{\sqrt{2\pi}}\int_{0}^{\infty}x^{-\frac{1}{2}}e^{-x(\frac{1}{2}+A)}\d x=e^{At}\left(\frac{1}{\sqrt{1+2A}}\right)^{N}.
\end{align*}
Taking $t=\frac{N}{2}$ and $A=\frac{1}{2}$ (so that $A-\log(1+2A)\leq-\frac{1}{6}$)
we obtain the stated result. 
\end{proof}
\begin{lem}
\label{lem:large-deviations} Suppose that $(a_{1},\ldots,a_{N})\in\R^{N}$
and there exist constants $A,C>0$ such that 
\begin{enumerate}
\item $\sum_{i=1}^{N}a_{i}=0$, 
\item $\sum_{i=1}^{N}a_{i}^{2}\leq C$, and 
\item $|a_{i}|\leq A$ for each $1\leq i\leq N$. 
\end{enumerate}
Then, 
\begin{itemize}
\item[(i)] If $X_{1},\ldots,X_{N}$ are independent $\chi_{1}^{2}$-distributed
random variables then for all $t>0$, 
\begin{align*}
\P\left(\left|\sum_{i=1}^{N}a_{i}X_{i}\right|\geq t\right)\leq2\left(\frac{At}{C}+1\right)^{\frac{C}{2A^{2}}}e^{-\frac{t}{2A}}.
\end{align*}
\item[(ii)] If $X_{1},\ldots,X_{N}$ are independent $\chi_{2}^{2}$-distributed
random variables then for all $t>0$, 
\begin{align*}
\P\left(\left|\sum_{i=1}^{N}a_{i}X_{i}\right|\geq t\right)\leq2\left(\frac{At}{2C}+1\right)^{\frac{C}{A^{2}}}e^{-\frac{t}{2A}}.
\end{align*}
\end{itemize}
\end{lem}

\begin{rem}
Note that condition (3) in Lemma \ref{lem:large-deviations} immediately
follows from condition (2) since we must have $|a_{i}|\leq\sqrt{C}$
for all $1\leq i\leq N$. Likewise, condition (2) follows from condition
(3) with $C=A^{2}N$. 
\end{rem}

\begin{proof}
We start with (i). Notice that 
\begin{align*}
\P\left(\left|\sum_{i=1}^{N}a_{i}X_{i}\right|\geq t\right)=\P\left(\sum_{i=1}^{N}a_{i}X_{i}\geq t\right)+\P\left(-\sum_{i=1}^{N}a_{i}X_{i}\geq t\right).
\end{align*}
Now for any $\varepsilon\in[0,\frac{1}{2A})$, exponential Chebyshev
inequality along with independence of the $X_{i}$ and the formula
\eqref{eq:pdfchidist} implies that 
\begin{align*}
 & \P\left(\sum_{i=1}^{N}a_{i}X_{i}\geq t\right)\\
 & \leq e^{-t\varepsilon}\E\left(\exp\left(\varepsilon\sum_{i=1}^{N}a_{i}X_{i}\right)\right) &  & =e^{-t\varepsilon}\prod_{i=1}^{N}\E\exp\left(\varepsilon a_{i}X_{i}\right)\\
 & =e^{-t\varepsilon}\prod_{i=1}^{N}\frac{1}{\sqrt{2\pi}}\int_{0}^{\infty}x^{-\frac{1}{2}}e^{-\frac{x}{2}(1-2\varepsilon a_{i})}\d x &  & =e^{-t\varepsilon}\prod_{i=1}^{N}\frac{1}{\sqrt{1-2\varepsilon a_{i}}}\\
 & =e^{-t\varepsilon}\exp\left(-\frac{1}{2}\sum_{i=1}^{N}\log(1-2\varepsilon a_{i})\right) &  & =e^{-t\varepsilon}\exp\left(\varepsilon\sum_{i=1}^{N}a_{i}+\frac{1}{2A^{2}}\sum_{i=1}^{N}a_{i}^{2}\sum_{n=2}^{\infty}\frac{(2A\varepsilon)^{n}\left(\frac{a_{i}}{A}\right)^{n-2}}{n}\right).
\end{align*}
The final equality follows from $|2\varepsilon a_{i}|<1$. Now by
assumption (3), we have $\left|\frac{a_{i}}{A}\right|^{n-2}\leq1$
and so using assumptions (1) and (2) we have 
\begin{align*}
\P\left(\sum_{i=1}^{N}a_{i}X_{i}\geq t\right) & \leq e^{-t\varepsilon}\exp\left(\frac{C}{2A^{2}}\sum_{n=2}^{\infty}\frac{(2A\varepsilon)^{n}}{n}\right)\\
 & =e^{-t\varepsilon}\exp\left(\log\left((1-2A\varepsilon)^{-\frac{C}{2A^{2}}}\right)-\frac{C\varepsilon}{A}\right)\\
 & =\frac{e^{-\varepsilon\left(t+\frac{C\varepsilon}{A}\right)}}{(1-2A\varepsilon)^{\frac{C}{2A^{2}}}},
\end{align*}
the second equality following from the fact that $|2A\varepsilon|<1$.
We now choose $\varepsilon\in[0,\frac{1}{2A})$ that minimizes this
upper bound. This can readily been seen to be given by 
\begin{align*}
\varepsilon=\frac{1}{2A}-\frac{\frac{C}{2A^{2}}}{t+\frac{C}{A}}\in\left[0,\frac{1}{2A}\right).
\end{align*}
We hence obtain the upper bound 
\begin{align*}
\P\left(\sum_{i=1}^{N}a_{i}X_{i}\geq t\right)\leq\left(\frac{At}{C}+1\right)^{\frac{C}{2A^{2}}}e^{-\frac{t}{2A}}.
\end{align*}
The same bound applies to $\P\left(-\sum_{i=1}^{N}a_{i}X_{i}\geq t\right)$
since we may set $b_{i}=-a_{i}$ and then $(b_{1},\ldots,b_{N})\in\R^{N}$
satisfies assumptions (1), (2) and (3) so that the above computations
still hold. 

The proof of (ii) follows identically but using the probability density
function $f_{2}(x)$ rather than $f_{1}(x)$. 
\end{proof}

\subsection{Random matrix estimates}
\begin{lem}
\label{lem:complex-random-matrices}Suppose that $V$ is an $n$-dimensional
complex Hermitian inner product space with orthonormal basis $\{e_{i}\}_{i=1}^{n}$,
and $u$ is a Haar random unitary matrix in $U(V)$. Then, 
\begin{enumerate}
\item For any fixed vector $\beta=\sum_{1\leq i,j\leq n}\beta_{ij}e_{i}\otimes\check{e_{j}}\in V\otimes\check{V}$
with $\beta_{ij}\in\C$, $\sum_{1\leq i,j\leq n}|\beta_{ij}|^{2}\leq C$
and $\sum_{i=1}^{n}\beta_{ii}=0$, for any $1\leq k\leq n$, and any
$T>0$ we have 
\[
\P_{u\in U(V)}\left(|\langle\beta,ue_{k}\otimes\widecheck{ue_{k}}\rangle|\geq T\right)\leq6e^{-\frac{nT}{32\sqrt{C}}}+2e^{-\frac{n}{6}}.
\]
\item For any fixed vector $\alpha=\sum_{1\leq i,j\leq n}\alpha_{ij}e_{i}\otimes e_{j}\in V\otimes V$
with $\alpha_{ij}\in\C$ and $\sum_{1\leq i,j\leq n}|\alpha_{ij}|^{2}\leq C$,
for any $1\leq k\leq n$, and any $T>0$ we have 
\[
\P_{u\in U(V)}\left(|\mathrm{Re}\langle\alpha,ue_{k}\otimes ue_{k}\rangle|\geq T\right)\leq6e^{-\frac{nT}{32\sqrt{C}}}+2e^{-\frac{n}{6}}.
\]
\end{enumerate}
\end{lem}

\begin{proof}
\emph{Proof of Part 1. }We have $|\langle\beta,ue_{k}\otimes\widecheck{ue_{k}}\rangle|=|\langle u^{-1}Mue_{k},e_{k}\rangle|$
where $M\in\End(V)$ is the operator defined by $M(e_{j})=\sum_{i}\beta_{ij}e_{i}$.
The conditions on $\beta$ imply that $M$ has zero trace and Hilbert-Schmidt
norm bounded by $\sqrt{C}$.

Write $M=H_{1}+iH_{2}$ where $H_{1}\eqdf\frac{1}{2}\left(M+M^{*}\right)$
and $H_{2}\eqdf\frac{1}{2i}\left(M-M^{*}\right)$ are Hermitian operators.
We have $\|H_{1}\|_{\HS}^{2}+\|H_{2}\|_{\HS}^{2}=\|M\|_{\HS}^{2}\leq C$
and hence
\[
\tr(H_{1})=\tr(H_{2})=0,\quad\|H_{1}\|_{\HS}^{2},\:\|H_{2}\|_{\HS}^{2}\leq C.
\]
Also,
\begin{equation}
\P\left(|\langle\beta,ue_{k}\otimes\widecheck{ue_{k}}\rangle|\geq T\right)\leq\sum_{i=1,2}\P\left(|\langle u^{-1}H_{i}ue_{k},e_{k}\rangle|\geq\frac{T}{2}\right).\label{eq:split-up}
\end{equation}
Since each $H_{i}$ is Hermitian, it is conjugate to a real diagonal
matrix $D_{i}$ with the same Hilbert-Schmidt norm and trace zero
by a unitary operator, and by bi-invariance of Haar measure, we obtain
\[
\P\left(|\langle u^{-1}H_{i}ue_{k},e_{k}\rangle|\geq\frac{T}{2}\right)\leq\P\left(|\langle u^{-1}D_{i}ue_{1},e_{1}\rangle|\geq\frac{T}{2}\right).
\]
We treat only $D_{1}$ as the bound for $D_{2}$ is the same. Thus
we can assume that $H_{1}=D_{1}=\mathrm{diag}(\lambda_{1},\ldots,\lambda_{\dim V})$
with 
\[
\sum_{i}\lambda_{i}=0,\quad\sum_{i}|\lambda_{i}|^{2}\leq C,
\]
and we have 
\[
|\langle u^{-1}D_{i}ue_{1},e_{1}\rangle|=\sum_{i}\lambda_{i}|u_{i1}|^{2}.
\]
As is well-known\footnote{A Haar random unitary matrix can be obtained by considering a random
matrix whose entries are i.i.d. standard complex normal random variables,
and then making the columns orthonormal by a Gram-Schmidt procedure
on the columns (see for example $\S\S$1.2 of \cite{Me19}). Carrying
out this algorithm starting with the 1\textsuperscript{st} column
just normalizes the column.} the entries $u_{i1}=\frac{1}{\sqrt{N}}\eta_{i}$ where $\eta_{i}$
are independent standard complex normal random variables and 
\[
N\eqdf\sum_{i=1}^{n}|\eta_{i}|^{2}=\frac{1}{2}\sum_{i=1}^{2n}Y_{i}
\]
 where $Y_{i}$ are independent $\chi_{1}^{2}$ random variables.
Hence by Lemma \ref{lem:large norm} 
\begin{equation}
\P\left(N\leq\frac{n}{2}\right)\leq e^{-\frac{n}{6}}.\label{eq:prob-denom}
\end{equation}
Thus with probability at least $1-e^{-\frac{n}{6}}$, we have $N\geq\frac{n}{2}$.
We have
\[
|\langle u^{-1}D_{i}ue_{1},e_{1}\rangle|=\frac{1}{2N}\left|\sum_{i=1}^{n}\lambda_{i}X_{i}\right|
\]
where $X_{i}$ are independent $\chi_{2}^{2}$ distributed random
variables, and so by Lemma \ref{lem:large-deviations} Part (ii) with
$C=C$ and $A=\sqrt{C}$
\begin{align}
\P\left(\left|\sum_{i=1}^{n}\lambda_{i}X_{i}\right|\geq\frac{nT}{4}\right)\leq\left(2+\frac{nT}{4\sqrt{C}}\right)e^{-\frac{nT}{8\sqrt{C}}}.\label{eq:prob-numerator}
\end{align}
Combining \eqref{eq:prob-denom} and \eqref{eq:prob-numerator} then
gives
\[
\P\left(|\langle u^{-1}D_{i}ue_{1},e_{1}\rangle|\geq\frac{T}{2}\right)\leq\left(2+\frac{nT}{4\sqrt{C}}\right)e^{-\frac{nT}{8\sqrt{C}}}+e^{-\frac{n}{6}}\leq3e^{-\frac{nT}{32\sqrt{C}}}+e^{-\frac{n}{6}}.
\]
Part 1 then follows from \eqref{eq:split-up}. 

\emph{Proof of Part 2.} This is similar except here we let $M\in\End(V)$
be the operator defined by $M(e_{j})=\sum_{i}\alpha_{ij}e_{i}$ and
write $M=S+R$ with $S\eqdf\frac{1}{2}(M+M^{T})$ and $R\eqdf\frac{1}{2}\left(M-M^{T}\right)$
where transpose is defined with respect to the real inner product
$\mathrm{Re}\langle\bullet,\bullet\rangle$. We have $u^{T}Ru=0$
so $R$ makes no contribution to $\langle$$\alpha,ue_{k}\otimes ue_{k}\rangle$.

The rest of the proof follows analogous lines to the proof of part
1, diagonalizing the real and imaginary parts of $S$ by orthogonal
(unitary) matrices. This leads to bounding $\P\left(\left|\mathrm{Re}\left(\sum_{i=1}^{n}\lambda_{i}\eta_{i}^{2}\right)\right|\geq\frac{nT}{4}\right)$
and $\P\left(\left|\mathrm{Im}\left(\sum_{i=1}^{n}\lambda'_{i}\eta_{i}^{2}\right)\right|\geq\frac{nT}{4}\right)$
where $\lambda_{i},\lambda'_{i}\in\R$, $\sum\lambda_{i}^{2},\sum(\lambda'_{i})^{2}\leq C$
and $\eta_{i}$, $1\leq i\leq n$ are independent standard complex
normals. For the first we have
\begin{align*}
\P\left(\left|\mathrm{Re}\left(\sum_{i=1}^{n}\lambda_{i}\eta_{i}^{2}\right)\right|\geq\frac{nT}{4}\right) & =\P\left(\left|\sum_{i=1}^{n}\lambda_{i}(x_{i}^{2}-y_{i}^{2})\right|\geq\frac{nT}{4}\right)\\
 & =\P\left(\left|\sum_{i=1}^{n}\lambda_{i}(X_{i}-Y_{i})\right|\geq\frac{nT}{2}\right)
\end{align*}
where this time, $X_{i}$ and $Y_{i}$ are independent $\chi_{1}^{2}$
distributed random variables. One can apply Lemma \ref{lem:large-deviations}
Part (i) with 
\[
a_{i}\eqdf\begin{cases}
\lambda_{i} & \text{if \ensuremath{i=1,\ldots,n},}\\
-\lambda_{i-n} & \text{if \ensuremath{i=n+1,\ldots,2n} }
\end{cases}
\]
 to obtain 
\[
\P\left(\left|\mathrm{Re}\left(\sum_{i=1}^{n}\lambda_{i}\eta_{i}^{2}\right)\right|\geq\frac{nT}{4}\right)\leq\left(2+\frac{nT}{2\sqrt{C}}\right)e^{-\frac{nT}{4\sqrt{C}}}.
\]
Dealing with $\P\left(\left|\mathrm{Im}\left(\sum_{i=1}^{n}\lambda_{i}'\eta_{i}^{2}\right)\right|\geq\frac{nT}{4}\right)$
is similar. These lead to the stated result.
\end{proof}
\begin{lem}
\label{lem:real-random-matrices}Suppose that $V$ is a real inner
product space with $n\eqdf\dim V$. Then for any fixed vector $\beta=\sum_{1\leq i,j\leq n}\beta_{ij}e_{i}\otimes\check{e_{j}}\in V\otimes\check{V}$,
$\sum_{1\leq i,j\leq n}|\beta_{ij}|^{2}\leq C$ and $\sum_{i=1}^{n}\beta_{ii}=0$,
for any $1\leq k\leq n$, and any $T>0$ we have 
\[
\P_{o\in O(V)}\left(|\langle\beta,oe_{k}\otimes\widecheck{oe_{k}}\rangle|\geq T\right)\leq6e^{-\frac{nT}{32\sqrt{C}}}+2e^{-\frac{n}{12}}.
\]
\end{lem}

\begin{proof}
This is just the real version of Lemma \ref{lem:complex-random-matrices}
Part 1. The proof is along exactly the same lines, using that the
first column of an orthogonal random matrix is obtained by choosing
independent standard real normal random variables as the entries,
and then normalizing. Accordingly, one ends up using Lemma \ref{lem:large-deviations}
Part (ii).
\end{proof}

\section{Proof of main results}

\label{sec:proof-main-thm} 

Let $(\pi,V)\in\hat{G}$ be an irreducible representation of $G$,
$1\leq j,k\leq n\eqdf\dim V$, and $f:G\to\R$ have zero mean. The
randomness of the basis enters into the error term given in Proposition
\ref{prop:errorboundreal} via the quantities
\begin{align*}
\left\langle \frac{\mathrm{dim}V}{|G|}\sum_{g\in G}f(g)(\pi\otimes\check{\pi})(g)(v_{k}^{V}\otimes\check{v}_{k}^{V}),w_{j}^{V}\otimes\check{w}_{j}^{V}\right\rangle ,\\
\left\langle \frac{\mathrm{dim}V}{|G|}\sum_{g\in G}f(g)(\pi\otimes\pi)(g)(v_{k}^{V}\otimes v_{k}^{V}),w_{j}^{V}\otimes w_{j}^{V}\right\rangle .
\end{align*}
Accordingly, let $v\eqdf v_{k}^{V}$, $e_{i}\eqdf e_{i}^{V}$ and 

\begin{align}
x\eqdf\frac{\mathrm{dim}V}{|G|}\sum_{g\in G}f(g)(\pi\otimes\check{\pi})(g)(v\otimes\check{v}),\quad & y\eqdf\frac{\mathrm{dim}V}{|G|}\sum_{g\in G}f(g)(\pi\otimes\pi)(g)(v\otimes v).\label{eq:defnofx}
\end{align}
We write 
\begin{align}
x\eqdf\sum_{i,j=1}^{n}x_{ij}e_{i}\otimes\check{e}_{j},\quad & y\eqdf\sum_{i,j=1}^{n}y_{ij}e_{i}\otimes e_{j}\label{eq:expansionofx}
\end{align}
for some $x_{ij},\,y_{ij}\in\C$. The vectors $x$ and $y$ satisfy
the following properties.
\begin{lem}
\label{lem:propertiesofx} Let $x$ be defined as in \eqref{eq:defnofx}
and \eqref{eq:expansionofx}. Then, 
\begin{enumerate}
\item[(i)] $\sum_{i,j=1}^{n}|x_{ij}|^{2},\sum_{i,j=1}^{n}|y_{ij}|^{2}\leq\frac{\|f\|_{\ell^{2}}^{2}\dim V}{|G|}$,
and 
\item[(ii)] $\sum_{i=1}^{n}x_{ii}=0$. 
\end{enumerate}
\end{lem}

\begin{proof}
Using \eqref{eq:expansionofx}, we see that $\sum_{i,j}|x_{ij}|^{2}=\left\langle x,x\right\rangle $,
and so computing this inner product with the expression \eqref{eq:defnofx},
we obtain
\begin{align*}
\sum_{i,j}|x_{ij}|^{2} & =\frac{(\dim V)^{2}}{|G|^{2}}\sum_{g}\sum_{h}f(g)f(h)\left\langle (\pi\otimes\check{\pi})(g)(v\otimes\check{v}),(\pi\otimes\check{\pi})(h)(v\otimes\check{v})\right\rangle \\
 & =\frac{(\dim V)^{2}}{|G|^{2}}\sum_{g}\sum_{h}f(g)f(h)\left|\left\langle \pi(g)v,\pi(h)v\right\rangle \right|^{2}\\
 & \leq\frac{(\dim V)^{2}}{|G|^{2}}\sum_{g}\sum_{h}\frac{|f(g)|^{2}+|f(h)|^{2}}{2}\left|\left\langle \pi(g)v,\pi(h)v\right\rangle \right|^{2}\\
 & =\frac{(\dim V)^{2}}{|G|}\sum_{g}|f(g)|^{2}\frac{1}{|G|}\sum_{h}\left|\left\langle \pi(g)v,\pi(h)v\right\rangle \right|^{2}\\
 & =\frac{\dim V}{|G|}\|f\|_{\ell^{2}}^{2},
\end{align*}
with the last equality following from Schur orthogonality. The same
bound holds for $y$ since 
\[
\sum_{i,j}|y_{ij}|^{2}=\frac{(\dim V)^{2}}{|G|^{2}}\sum_{g}\sum_{h}f(g)f(h)\left(\left\langle \pi(g)v,\pi(h)v\right\rangle \right)^{2}.
\]
To prove (ii), we see from \eqref{eq:expansionofx} that $\sum_{i}x_{ii}=\left\langle x,\sum_{i}e_{i}\otimes\check{e_{i}}\right\rangle $.
Computing this inner product with \eqref{eq:defnofx} we obtain
\begin{align*}
\sum_{i}x_{ii} & =\frac{\dim V}{|G|}\sum_{g}f(g)\sum_{i}\left\langle \pi(g)v,e_{i}\right\rangle \left\langle \check{\pi}(g)\check{v},\check{e_{i}}\right\rangle \\
 & =\frac{\dim V}{|G|}\sum_{g}f(g)\left\langle \pi(g)v,\sum_{i}\left\langle \pi(g)v,e_{i}\right\rangle e_{i}\right\rangle \\
 & =\frac{\dim V}{|G|}\sum_{g}f(g)\left\langle \pi(g)v,\pi(g)v\right\rangle \\
 & =\frac{\dim V}{|G|}\sum_{g}f(g)=0,
\end{align*}
since $f$ has mean zero.
\end{proof}
The following bound applies to the error terms that arise from random
basis elements coming from complex non-self-dual or quaternionic representations
(type 1 or type 2 in the previous language) in Proposition \ref{prop:errorboundreal}.
\begin{prop}
\label{prop:unitary} Let $(\pi,V)\in\hat{G}$ be an irreducible representation
of $G$ that is either complex non-self-dual or quaternionic. Then,
for any $t>0$ and indices $1\leq i,j\leq n\eqdf\dim V$,
\begin{align*}
\P_{u\in U(V)}\left(\left|\mathrm{Re}\left\langle \frac{\mathrm{dim}V}{|G|}\sum_{g\in G}f(g)(\pi\otimes\check{\pi})(g)(v_{i}^{V}\otimes\check{v}_{i}^{V}),ue_{j}^{V}\otimes\widecheck{ue_{j}^{V}}\right\rangle \right|\geq t\frac{\|f\|_{\ell^{2}}}{2\sqrt{|G|}}\right)\\
 & \hspace{-2cm}\leq6e^{-\frac{t\sqrt{\dim V}}{64}}+2e^{-\frac{\dim V}{6}},
\end{align*}
and 
\begin{align*}
\P_{u\in U(V)}\left(\left|\mathrm{Re}\left\langle \frac{\mathrm{dim}V}{|G|}\sum_{g\in G}f(g)(\pi\otimes\pi)(g)(v_{i}^{V}\otimes v_{i}^{V}),ue_{j}^{V}\otimes ue_{j}^{V}\right\rangle \right|\geq t\frac{\|f\|_{\ell^{2}}}{2\sqrt{|G|}}\right)\\
 & \hspace{-2cm}\leq6e^{-\frac{t\sqrt{\dim V}}{64}}+2e^{-\frac{\dim V}{6}}.
\end{align*}
\end{prop}

\begin{proof}
This follows by combining the respective parts of Lemma \ref{lem:complex-random-matrices}
and Lemma \ref{lem:propertiesofx}, with $C=\frac{\|f\|_{\ell^{2}}^{2}\dim V}{|G|}$
and $T=t\frac{\|f\|_{\ell^{2}}}{2\sqrt{|G|}}.$
\end{proof}
The next bound applies to the other error terms coming from real representations.
\begin{prop}
\label{prop:orthogonal}Let $(\pi,V)\in\hat{G}$ be a self-dual real
irreducible representation of $G$ and $1\leq i,j\leq n\eqdf\dim V$,
then for any $t>0,$
\begin{align*}
\P_{o\in O(V_{J})}\left(\left|\mathrm{Re}\left\langle \frac{\mathrm{dim}V}{|G|}\sum_{g\in G}f(g)(\pi\otimes\check{\pi})(g)(v_{i}^{V}\otimes\check{v}_{i}^{V}),oe_{j}^{V}\otimes\widecheck{oe_{j}^{V}}\right\rangle \right|\geq t\frac{\|f\|_{\ell^{2}}}{2\sqrt{|G|}}\right)\\
 & \hspace{-2cm}\leq6e^{-\frac{t\sqrt{\dim V}}{64}}+2e^{-\frac{\dim V}{12}}.
\end{align*}
\end{prop}

\begin{proof}
Let
\begin{align*}
x\eqdf\frac{\mathrm{dim}V}{|G|}\sum_{g\in G}f(g)(\pi\otimes\check{\pi})(g)(v_{i}^{V}\otimes\check{v}_{i}^{V}).
\end{align*}
Expanding $x$ over the basis $\{e_{i}^{V}\otimes\check{e}_{j}^{V}\}_{i,j}$
of $V\otimes\check{V}$ we obtain 
\begin{align*}
x=\sum_{i,j=1}^{n}x_{ij}e_{i}^{V}\otimes\check{e}_{j}^{V},
\end{align*}
for some $x_{ij}\in\C$. Because in this case, the inner product is
extended from a real inner product on the real subspace $V_{J}$ (cf.
$\S\S\S$\ref{subsec:Real-representations.}), and all $oe_{j}\in V_{J}$,
we have 
\[
\mathrm{Re}\left\langle x,oe_{j}^{V}\otimes\widecheck{oe_{j}^{V}}\right\rangle =\langle\beta,oe_{j}^{V}\otimes\widecheck{oe_{j}^{V}}\rangle
\]
where $\beta=\sum_{i,j=1}^{n}\beta_{ij}e_{i}^{V}\otimes\check{e}_{j}^{V}$,
$\beta_{ij}\eqdf\mathrm{Re}(x_{ij})$. We thus have $\sum_{ij}|\beta_{ij}|^{2}\leq\sum_{ij}|x_{ij}|^{2}\leq\frac{\|f\|_{2}^{2}\dim V}{|G|}$
and $\sum_{i}\beta_{ii}=0$ using Lemma \ref{lem:propertiesofx}.
We can apply Lemma \ref{lem:real-random-matrices} to get the result.
\end{proof}
We are now ready to combine the probabilistic estimates of Propositions
\ref{prop:unitary} and \ref{prop:orthogonal} with the deterministic
error estimate of Proposition \ref{prop:errorboundreal} to prove
Theorem \ref{thm:main-theorem-random}.
\begin{proof}[Proof of Theorem \ref{thm:main-theorem-random}]
Recall the probability space $(X,\P)$ and the notation used for
the elements of the random eigenbasis constructed in $\S$\ref{sec:Random-basis-construction}.
For each $k=1,\ldots,M,$ we define $\tilde{f_{k}}=f_{k}-\frac{1}{|G|}\sum_{g\in G}f_{k}(g)$
and set
\begin{align*}
F_{1}(\pi,i,j,f_{k}) & \eqdf\left|\mathrm{Re}\left\langle \frac{\mathrm{dim}V}{|G|}\sum_{g\in G}\tilde{f_{k}}(g)(\pi\otimes\pi)(g)(v_{i}^{V}\otimes v_{i}^{V}),ue_{j}^{V}\otimes ue_{j}^{V}\right\rangle \right|,\\
F_{2}(\pi,i,j,f_{k}) & \eqdf\left|\mathrm{Re}\left\langle \frac{\mathrm{dim}V}{|G|}\sum_{g\in G}\tilde{f_{k}}(g)(\pi\otimes\check{\pi})(g)(v_{i}^{V}\otimes\check{v}_{i}^{V}),ue_{j}^{V}\otimes\widecheck{ue_{j}^{V}}\right\rangle \right|,\\
F_{3}(\pi,i,j,f_{k}) & \eqdf\left|\mathrm{Re}\left\langle \frac{\mathrm{dim}V}{|G|}\sum_{g\in G}\tilde{f_{k}}(g)(\pi\otimes\check{\pi})(g)(v_{i}^{V}\otimes\check{v}_{i}^{V}),oe_{j}^{V}\otimes\widecheck{oe_{j}^{V}}\right\rangle \right|,
\end{align*}
where in $F_{1}$ and $F_{2}$ we assume $V$ is not real and in $F_{3}$
we assume that $V$ is a real representation. In all cases, we may
assume that $\pi$ is non-trivial since this is a one dimensional
representation with corresponding eigenspace spanned by the constant
function for which the desired estimates trivially hold.

Let $\mathcal{E}_{t}$ denote the event that some $F_{1}(\pi,i,j,f_{k})$
or $F_{2}(\pi,i,j,f_{k})$ with $(\pi,V)$ complex or quaternionic,
or some $F_{3}(\pi,i,j,f_{k})$ with $(\pi,V)$ real satisfies 
\[
F_{\ell}(\pi,i,j,f_{k})>t\frac{\|\tilde{f}_{k}\|_{\ell^{2}}}{2\sqrt{|G|}}.
\]
By carrying out a union bound over all $\pi\in\hat{G}-\mathrm{triv}$,
all functions $f_{1},\ldots,f_{M}$ in the collection, and $1\leq i,j\leq\dim V$
with the estimates from Propositions \ref{prop:unitary} and \ref{prop:orthogonal},
we obtain
\begin{align}
\P(\mathcal{E}_{t}) & \leq2M\sum_{(\pi,V)\in\hat{G}-\mathrm{triv}}(\dim V)^{2}\left(6e^{-\frac{t\sqrt{\dim V}}{64}}+2e^{-\frac{\dim V}{12}}\right)\label{eq:proba-upper-bound}
\end{align}
Now assume we have a basis $\B\subset\ell^{2}(V(G))$ that is not
in $\mathcal{E}_{t}$, and let $\varphi\in\B$. Then, for any of the
functions $f_{k},$ since $\|\tilde{f}_{k}\|_{2}\leq\|f\|_{2}$ we
obtain from Proposition \ref{prop:errorboundreal} that
\begin{align*}
\left|\sum_{g\in G}f_{k}(g)|\varphi(g)|^{2}-\frac{1}{|G|}\sum_{g\in G}f_{k}(g)\right| & \leq t\frac{\|f_{k}\|_{\ell^{2}}}{\sqrt{|G|}}.
\end{align*}
This completes the proof.
\end{proof}
The proof of Theorem \ref{thm:main-thm-there-exists} is then immediate. 
\begin{rem}
\label{rem:complex-eiengbases} The proofs of Theorems \ref{thm:main-thm-there-exists}
and \ref{thm:main-theorem-random} simplify if one only wishes to
consider complex-valued eigenbases. The construction of these bases
is similar to $\S$\ref{sec:Random-basis-construction}. Indeed, given
an irreducible representation $(\pi,V)$ of $G$, let $\{v_{i}^{V}\}$
be an orthonormal basis of $V$ consisting of $\pi(A)$ eigenvectors
and let $\{w_{j}^{V}\}$ be any other orthonormal basis of $V$. The
collection 
\begin{align*}
\left\{ \check{w}_{j}^{V}\otimes v_{i}^{V}:i,j=1,\ldots,\dim V\right\} 
\end{align*}
forms an orthonormal basis of $\check{V}\otimes V$corresponding to
the following orthonormal adjacency operator eigenfunctions 
\[
\varphi_{i,j}^{V}(g)\eqdf\frac{\sqrt{\dim V}}{\sqrt{|G|}}\langle\pi(g)v_{i}^{V},w_{j}^{V}\rangle
\]
in $\ell^{2}(G)$. To randomize this basis, we fix an orthonormal
basis $\{e_{j}^{V}\}_{j}$ of $V$ and then given a Haar random unitary
operator $u\in U(V)$, we set $w_{j}^{V}=ue_{j}^{V}$ for each $j=1,\ldots,\dim V$.

The upper bound obtained in Proposition \ref{prop:errorboundreal}
for type 3 basis elements then holds for the collection $\{\varphi_{i,j}^{V}\}$
but with the real part in the upper bound replaced by the absolute
value; the proof of this is analogous. Expanding the vectors in the
inner product for this upper bound as in $\S$\ref{sec:proof-main-thm},
we recover Lemma \ref{lem:propertiesofx} identically. Thus, we can
combine Lemma \ref{lem:propertiesofx} and part 1 of Lemma \ref{lem:complex-random-matrices}
to prove the same probabilistic bound in the first part of Proposition
\ref{prop:unitary} (without the real part) for the $\varphi_{i,j}^{V}$.
Theorems \ref{thm:main-thm-there-exists} and \ref{thm:main-theorem-random}
then follow via a union bound over the irreducible representations
and basis vectors as in the proof of Theorem \ref{thm:main-theorem-random}.
In fact, in the complex-valued basis case, one may take the functions
to be complex-valued.
\end{rem}

\begin{proof}[Proof of Corollary \ref{cor:que-large-qr}]
 We use Theorem \ref{thm:main-thm-there-exists} with $t=64(\varepsilon+1)\frac{\log(|G|)}{\sqrt{\mathfrak{D}(G)}}$.
Then, 
\[
2M\sum_{(\pi,V)\in\hat{G}-\mathrm{triv}}(\dim V)^{2}\left(6e^{-\frac{t\sqrt{\dim V}}{64}}+2e^{-\frac{\dim V}{12}}\right)<12M|G|e^{-(\varepsilon+1)\log(|G|)}+4M|G|e^{-\frac{\mathfrak{D}(G)}{12}},
\]
and so requiring that both terms in this summation are less than $\frac{1}{2}$
gives the required bound on $M$ for a basis satisfying \eqref{eq:large_qr_bound}
to exist.
\end{proof}
\begin{proof}[Proof of Corollary \ref{cor:que-small-scale}]
 Since the collection of subsets $A_{i}$ satisfy the bound $c|G|^{1-\eta}\leq|A_{i}|\leq C|G|^{1-\eta}$
on their size, there are at most $\frac{1}{c}|G|^{\eta}$ of them.
We take $t=\frac{128\log|G|}{\sqrt{\mathfrak{D}(G)}}$ so that when
$|G|$ is sufficiently large (dependent only upon $c$ and $\eta$),
we have $12\frac{1}{c}|G|^{\eta-1}+4\frac{1}{c}|G|^{\eta+1}e^{-\frac{1}{12}|G|^{\eta+\varepsilon}}<1$.
Thus by Theorem \ref{thm:main-thm-there-exists} if one takes the
functions to be the at most $\frac{1}{c}|G|^{\eta}$ indicator functions
on the sets $A_{i}$, there exists an orthonormal eigenbasis $\mathcal{B}$
such that

\[
\left|\mu_{\varphi}[A_{i}]-\frac{|A_{i}|}{|G|}\right|\leq\frac{128\log|G|}{|G|^{\frac{1}{2}\eta+\frac{1}{2}\varepsilon}}\frac{\sqrt{|A_{i}|}}{\sqrt{|G|}}\leq\frac{128\log|G|}{\sqrt{c}|G|^{\frac{1}{2}\varepsilon}}\frac{|A_{i}|}{|G|},
\]
for every $\varphi\in\mathcal{B}$ and each set $A_{i}$, with the
last inequality following from $\sqrt{|A_{i}|}\leq\frac{|A_{i}|}{\sqrt{|A_{i}|}}\leq\frac{|A_{i}|}{\sqrt{c}|G|^{\frac{1}{2}-\frac{1}{2}\eta}}$.
\end{proof}
\begin{proof}[Proof of Proposition \ref{prop:que-perm-groups}]
 For $\mathrm{Sym}(n)$, we firstly note that the sign representation
is one-dimensional and the corresponding eigenfunctions are spanned
by the function assigning $1$ to even permutations and $-1$ to odd
permutations and thus these eigenfunctions already satisfy the QUE
bound exactly after normalization. When doing the randomization in
the proof of Theorem \ref{thm:main-theorem-random}, we thus only
require the union bound to run over the non-sign and non-trivial permutations.
In other words, for $\mathrm{Sym}(n)$, Theorem \ref{thm:main-thm-there-exists}
holds when 
\[
2M_{n}\sum_{(\pi,V)\in\widehat{\mathrm{Sym}(n)}-\mathrm{\{triv},\mathrm{\,sign}\}}(\dim V)^{2}\left(6e^{-\frac{t_{n}\sqrt{\dim V}}{64}}+2e^{-\frac{\dim V}{12}}\right)<1,
\]
instead. Now, consider $t_{n}=192\frac{\log(n-1)}{\sqrt{n-1}}$. Since
$\dim V\geq n-1$ for all non-trivial and non-sign irreducible representations
$(\pi,V)$ we have that $e^{-\frac{t_{n}\sqrt{\dim V}}{64}}\leq e^{-3\log(\dim V)}$.
Moreover, for $n\geq24$ and $(\pi,V)$ non-sign and non-trivial we
have $(\dim V)^{2}e^{-\frac{\dim V}{12}}\leq(\dim V)^{-1}n^{3}e^{-\frac{n}{12}}.$
It follows that 
\begin{align}
 & 2M_{n}\sum_{(\pi,V)\in\widehat{\mathrm{Sym}(n)}-\mathrm{\{triv},\mathrm{sign}\}}(\dim V)^{2}\left(6e^{-\frac{t_{n}\sqrt{\dim V}}{64}}+2e^{-\frac{\dim V}{12}}\right)\nonumber \\
< & 2M_{n}\left(\left(\sum_{(\pi,V)\in\widehat{\mathrm{Sym}(n)}}(\dim V)^{-1}\right)-2\right)\left(6+2n^{3}e^{-\frac{n}{12}}\right).\label{eq:probbd}
\end{align}
The quantity $\sum_{(\pi,V)\in\widehat{\mathrm{Sym}(n)}}(\dim V)^{-1}$
is precisely the Witten Zeta function of the symmetric group at $1$.
By \cite{Lu96,Mu.Pu02,Li.Sh04,Ga06} it is known that 
\[
\sum_{(\pi,V)\in\widehat{\mathrm{Sym}(n)}}(\dim V)^{-1}=2+O(n^{-1})
\]
 and so \eqref{eq:probbd} is $O(M_{n}(n^{-1}+n^{2}e^{-\frac{n}{12}}))$.
Thus, taking $M_{n}=o_{n\to\infty}(n)$ is sufficient for the existence
of a basis satisfying \eqref{eq:permutation_bound}.

In the case of $\mathrm{Alt}(n)$, we note that any irreducible representation
corresponds to two irreducible representations of $\mathrm{Sym}(n)$
and so $\sum_{(\pi,V)\in\widehat{\mathrm{Alt}(n)}}(\dim V)^{-1}=1+O(n^{-1})$.
In addition, $\mathfrak{D}(\mathrm{Alt}(n))\geq n-1$ and so an identical
argument to the case for $\mathrm{Sym}(n)$ (this time there is no
sign representation) gives the same result for $\mathrm{Alt}(n)$.
\end{proof}

\section*{Acknowledgements}

MM and JT were supported by funding from the European Research Council
(ERC) under the European Union’s Horizon 2020 research and innovation
programme (grant agreement No 949143). YZ was supported by funding
from NSF CAREER award DMS-2044606, a Sloan Research Fellowship, and
the MIT Solomon Buchsbaum Fund. We thank Sean Eberhard for useful
discussions about this project. We thank Assaf Naor, Ashwin Sah, and
Mehtaab Sawhney for their comments on a previous version of this paper
that helped us improve it.

{\footnotesize{}\bibliographystyle{amsalpha}
\bibliography{QUECayleyGraphs}
}{\footnotesize\par}

\noindent Michael Magee, \\
Department of Mathematical Sciences,\\
Durham University, \\
Lower Mountjoy, DH1 3LE Durham,\\
United Kingdom

\noindent \texttt{michael.r.magee@durham.ac.uk}~\linebreak{}

\noindent Joe Thomas, \\
Department of Mathematical Sciences,\\
Durham University, \\
Lower Mountjoy, DH1 3LE Durham,\\
United Kingdom

\noindent \texttt{joe.thomas@durham.ac.uk}~\linebreak{}

\noindent Yufei Zhao, \\
Department of Mathematics,\\
Massachusetts Institute of Technology, \\
Cambridge, MA 02139,\\
USA

\noindent \texttt{yufeiz@mit.edu}
\end{document}